\documentclass[secnum,leqno]{rims-bessatsu}
\usepackage{amssymb, amsmath}
\usepackage{graphics}
\usepackage[dvips]{graphicx}
\usepackage{eepic}
\usepackage{epic}
\usepackage{amscd}
\usepackage{graphicx, color}

\newtheorem{thm}{Theorem}[section]
\newtheorem{crl}[thm]{Corollary}

\newtheorem{lmm}[thm]{Lemma}

\newtheorem{prp}[thm]{Proposition}

\theoremstyle{definition}

\newtheorem{conj}[thm]{\bf Conjecture}
\theoremstyle{remark}
\newtheorem{rem}{Remark}
\usepackage[all]{xy}

\newcommand{\Hol}{\mbox{{\rm Hol}}}
\newcommand{\Alg}{\mbox{{\rm Alg}}}

\newcommand{\Z}{\Bbb Z}
\newcommand{\C}{\Bbb C}
\newcommand{\R}{\Bbb R}

\newcommand{\K}{\Bbb K}

\newcommand{\RP}{\Bbb R\mbox{{\rm P}}}
\newcommand{\KP}{\Bbb K\mbox{{\rm P}}}

\newcommand{\Map}{\mbox{{\rm Map}}}

\newcommand{\CP}{\Bbb C {\rm P}}

\newcommand{\p}{\prime}

\newcommand{\I}{\mbox{{\rm (i)}}}
\newcommand{\II}{\mbox{{\rm (ii)}}}

\numberwithin{equation}{section}

\title{
Spaces of equivariant algebraic maps from real projective spaces
into complex projective spaces}
\TitleHead{Spaces of equivariant algebraic maps into complex projective spaces}          
\author{Andrzej \textsc{Kozlowski}
\footnote{Institute of Applied Mathematics and Mechanics,University of Warsaw
Banacha 2, 02-097 Warsaw,
Poland
\newline e-mail: \texttt{akoz@mimuw.edu.pl}}
          ~and Kohhei \textsc{Yamaguchi}\footnote{
          University of Electro-Communications, Chofu Tokyo 182-8585, Japan
\endgraf e-mail: \texttt{kohhei@im.uec.ac.jp}}}
\AuthorHead{A. Kozlowski and K. Yamaguchi
}         
\support{The second author is partially supported by Grant-in-Aid for Scientific Research
(No. 23540079), The Ministry of Education, Culture, Sports Science and Technology, Japan}  
\VolumeNo{x}           
\YearNo{201x}           
\PagesNo{000--000}      
\communication{Received April 20, 201x.
Revised September 11, 201x.}      
\begin{document}
%

\maketitle

\begin{abstract}      
We study the homotopy types of certain spaces closely related to the spaces of algebraic 
(rational) maps
from the $m$ dimensional real projective space into
the $n$ dimensional complex projective space for
$2\leq m\leq 2n$ (we conjecture this relation to be a homotopy equivalence).
In \cite{KY3} we proved  that natural maps of these spaces to the spaces of all continuous maps are
$\Z /2$-equivariant homotopy equivalences, where the $\Z/2$-equivariant action is
induced from the conjugation on $\C$.
In the same article we also proved  that
the homotopy types of the terms of the natural  degree  filtration approximate closer 
and closer  the homotopy type
of the space of continuous maps  and  obtained bounds that describe the closeness of the approximation in terms of the degrees of the maps.
In this paper, we improve the estimates of the bounds by using new methods invented in \cite{Mo3} and used in \cite{KY4}.

In addition, in the the last section, 
we prove a special case ($m=1$) of the conjecture
stated in \cite{AKY1} that our spaces are homotopy equivalent to the spaces of algebraic maps. 

\end{abstract}

\section{Introduction.}\label{section 1}
\paragraph{Summary of the contents.}

 In \cite{Mo2} Mostovoy showed (modulo certain errors that were corrected in \cite{Mo3}) that 
 if $2\leq m \leq n$ then the space  $\Hol (\CP^m,\CP^n)$ of holomorphic maps 
 from $\CP^m$ to $\CP^n$ has the same homotopy type of the space 
 $\Map (\CP^m,\CP^n)$ of coresponding continuous maps
up to a certain dimension, 
 which generalizes the  classical result of Segal \cite{Se} for $m=1$.
  In \cite{AKY1} we used a variant of Mostovoy's method to the analogous problem for  
  the space $\Alg (\RP^m,\RP^n)$ of algebraic maps from $\RP^m$ to $\RP^n$ for $2\leq m<n$.
  These our results can also be seen as generalizations of the results of \cite{GKY2} and \cite{Mo1}.
  In \cite{KY3} we used analogous methods to prove a homotopy 
  (or homology) approximation theorem for the space
  $\Alg (\RP^m,\CP^n)$ of real algebraic maps from $\RP^m$ to $\CP^n$ 
  when $2\leq m\leq 2n$. Combing this result with the main theorem of  \cite{AKY1},  
  we obtained a  $\Bbb Z/2$-equivariant homotopy approximation 
  result in \cite{KY3} (where the $\Bbb Z/2$-action on $\CP^n$ is induced by complex conjugation),   
  which is itself a generalization of  \cite[Theorem 3.7]{GKY2}. 
  
In his recent paper \cite{Mo3},  Mostovoy, in addition to correcting the mistakes in \cite{Mo2}, introduced a new idea that allowed him to improve the bounds on the degree of homotopy groups in his homotopy approximation theorem. With the help of analogous methods and some additional techniques
of ours,
in the recent work \cite{KY4}, we improved the bounds of the homotopy approximation theorem 
given in \cite{AKY1}. 
\par\vspace{3mm}\par
The first purpose of this article is to obtain a  improved version of the equivariant homotopy approximation theorem of \cite{KY3}. Since the arguments are analogous to those in \cite{KY4}, we only state our new results and refer to the references for detailed proofs 
(cf. Remark \ref{Remark: proof} below).

The second purpose is to  consider the possibility of approximating the space of continuous maps by its subspaces of algebraic maps of a fixed degree, analogously to the results of Segal \cite{Se} and Mostovoy \cite{Mo3}. 
Because the situation is analogous, in this paper
we only consider the space of algebraic maps from $\RP^m$ to $\RP^n$.
For this purpose we only need to prove that the  projection maps
$$
\Psi_d:A_d(m,n)\to \Alg_d^*(\RP^m,\RP^n), \quad\Gamma_d:\tilde{A}_d(m,n)\to\Alg_d(\RP^m,\RP^n)
$$ are homotopy equivalences. 
 It is easy to see that both of these maps have contractible fibres, however that alone is not sufficient to prove that the maps are homotopy equivalencies. 
 In fact, in \cite[Conjecture 3.8]{AKY1}, 
 we conjectured that this is true in general. Here we provide an informal argument that shows that this is true for $m=1$.  A proof of this fact already appears in   \cite[Proposition 2.1]{Mo1}. However we find that the argument given there is unconvincing since it does not seem to make any use of any properties of the map $\Psi_d$ (e.g. the fact that is a quasi-fibration).



\par\vspace{2mm}\par
In the remainder of this section we briefly  describe the notation  and  the definitions.
\paragraph{Notation.}
The notation of this paper is essentially analogous to the one used in \cite{AKY1}
and \cite{KY3}. Note that the presence of $\Bbb C$ indicates that a complex case is being considered (i.e. maps take values in $\CP^n$ rather than $\RP^n$ or polynomials have coefficients in $\Bbb C$ rather than $\Bbb R$, etc.).
\par
Let $m$ and $n$ be positive integers such that
$1\leq m \leq 2n$.
We choose  ${\bf e}_m=[1:0:\cdots :0]\in \RP^m$ and
${\bf e}_n^{\p}=[1:0:\cdots :0]\in\Bbb \CP^n$
as the base points of $\RP^m$ and  $\Bbb  \CP^n$, respectively.
Let $\Map^*(\RP^m,\Bbb  \CP^n)$ denote the space
consisting of all based maps
$f:(\RP^m,{\bf e}_{m})\to (\Bbb \CP^n,{\bf e}^{\p}_n)$. 
When $m\geq 2$,
we denote
by $\Map_{\epsilon}^* (\RP^m,\CP^n)$
the corresponding path component
of $\Map^* (\RP^m,\CP^n)$
for each
$\epsilon \in \Z/2=\{0,1\}=
\pi_0(\Map^*(\RP^m,\Bbb  \CP^n))$
(\cite{CS}). 
Similarly, let $\Map (\RP^m,\CP^n)$ denote the space
of all free maps $f:\RP^m\to\CP^n$ and
$\Map_{\epsilon}(\RP^m,\CP^n)$ the corresponding path component
of $\Map (\RP^m,\CP^n)$.
\par
\par\vspace{2mm}\par
A  map $f:\RP^m\to \CP^n$ is called a
{\it algebraic map of degree }$d$ if it can be represented as
a rational map of the form
$f=[f_0:\cdots :f_n]$ such that
 $f_0,\cdots ,f_n\in\Bbb C [z_0,\cdots ,z_m]$ are homogeneous polynomials of the same degree $d$
with no common {\it real } roots except ${\bf 0}_{m+1}=(0,\cdots ,0)\in\R^{m+1}$.
 We  denote by $\Alg_d(\RP^m,\CP^n)$ (resp. $\Alg_d^*(\RP^m,\CP^n)$) the space
 consisting of all (resp. based) algebraic maps $f:\RP^m\to \CP^n$
 of degree $d$.
 It is easy to see that there are inclusions 
 $$
 \Alg_d(\RP^m,\CP^n)\subset \Map_{[d]_2}(\RP^m,\CP^n),
\quad
\Alg^*_d(\RP^m,\CP^n)\subset \Map^*_{[d]_2}(\RP^m,\CP^n),
$$
 where $[d]_2\in\Z/2=\{0,1\}$ denotes the integer $d$ mod $2$.
\par
Let
$ A_{d}(m,n)(\Bbb C)$ denote  the space consisting of all $(n+1)$-tuples 
$(f_0,\cdots ,f_n)\in \Bbb C[z_0,\cdots ,z_m]^{n+1}$
of  homogeneous polynomials of degree $d$  with coefficients in $\Bbb C$  and without non-trivial common real roots
(but possibly with non-trivial common {\it non-real} ones).
Since $\C^*$ acts on $A_d(m,n)(\C)$ freely, one can define the projectivisation
$\tilde{A}_d^{\C}(m,n)$ by the orbit space
$\tilde{A}_d^{\C}(m,n)=A_d(m,n)(\C)/\C^*$.
\par  
Let $ A_{d}^{\Bbb C}(m,n)\subset A_d(m,n)(\Bbb C)$ be the subspace consisting of 
all $(n+1)$-tuples 
$(f_0,\cdots ,f_n)\in A_d(m,n)(\Bbb C)$
such that the coefficient of $z_0^d$ in $f_0$ is 1 and $0$ in the other $f_k$'s
$(k\not= 0$).
Then there are natural  projection maps
\begin{equation*}\label{Psi}
\Psi_d^{\Bbb C}:A_d^{\Bbb C}(m,n) \to \Alg_d^*(\RP^m,\CP^n),
\quad
\Gamma_d^{\C}:\tilde{A}_d^{\C}(m,n)\to \Alg_d(\RP^m,\CP^n).
\end{equation*}
\par
For $m\geq 2$ and $g\in \Alg_d^*(\RP^{m-1},\CP^n)$ a fixed algebraic map,
we denote by
$\Alg_d^{\Bbb C}(m,n;g)$ and $F(m,n;g)$  the spaces defined by
$$
\begin{cases}
\Alg_d^{\Bbb C}(m,n;g) &=\ \{f\in \Alg_d^*(\RP^m,\CP^n):f\vert \RP^{m-1}=g\},
\\
F(m,n;g) & =\ \{f\in \Map_{[d]_2}^*(\RP^m,\CP^n):f\vert \RP^{m-1}=g\}.
\end{cases}
$$
It is well-known that there is a homotopy equivalence
$F(m,n;g)\simeq  \Omega^m\CP^n$ (\cite{Sasao}).
\par
Let ${\cal H}^m_d$ denote the space of all
homogenous polynomials $h\in\C [z_0,\cdots ,z_m]$ of degree
$d$. We choose a fixed tuple ${\rm g}=(g_0,\cdots ,g_n)\in A_d^{\C}(m-1,n)$ such that
$\Psi_d^{\C}({\rm g})=g$. 
In this situation, we denote by
$A_d^*(\C)\subset ({\cal H}^m_{d})^{n+1}$  the subspace
given by
$A_d^*(\C):=\Big\{(g_0+z_mh_0,\cdots ,g_n+z_mh_n):h_k\in {\cal H}^m_{d-1}
\ (0\leq k\leq n)\Big\}$, and  define
the subspace $A_d^{\Bbb C}(m,n;g)\subset A_d^{\Bbb C}(m,n)$  
by 
$
A_d^{\Bbb C}(m,n;g)=
A_d^{\Bbb C}(m,n)\cap A^*_d(\C).
$
\par
Because
$\Psi_d^{\C}(f_0,\cdots ,f_n)\in \Alg_d^{\C}(m,n;g)$ for
any
$(f_0,\cdots ,f_n)\in A_d^{\C}(m,n;g)$,
one can define the projection
${\Psi_d^{\C}}^{\p}
:A^{\C}_d(m,n;g)\to \Alg_d^{\C}(m,n;g)$
by the restriction 
${\Psi_d^{\C}}^{\p}=\Psi_d^{\C}\vert A^{\C}_d(m,n;g).$
Let 
\begin{equation}
\begin{cases}
i_{d,\Bbb C}:\Alg_d^*(\RP^m,\CP^n)\stackrel{\subset}{\rightarrow} \Map_{[d]_2}^*(\RP^m,\CP^n),
\\
j_{d,\Bbb C}:\Alg_d(\RP^m,\CP^n)\stackrel{\subset}{\rightarrow} \Map_{[d]_2}(\RP^m,\CP^n),
\\
i_{d,\Bbb C}^{\p}:\Alg_d^{\Bbb C}(m,n;g)\stackrel{\subset}{\rightarrow} F(m,n;g)\simeq \Omega^m\CP^n
\end{cases}
\end{equation}
denote the inclusions and
\begin{equation}
\begin{cases}
i_d^{\Bbb C}=i_{d,\Bbb C}\circ \Psi_d^{\Bbb C}:A_d^{\Bbb C}(m,n)\to \Map_{[d]_2}^*(\RP^m,\CP^n),
\\
j_d^{\Bbb C}=j_{d,\Bbb C}\circ \Gamma_d^{\Bbb C}:\tilde{A}_d^{\Bbb C}(m,n)\to \Map_{[d]_2}^*(\RP^m,\CP^n),
\\
i_d^{\p\p}=i^{\p}_{d,\C}\circ {\Psi_d^{\C}}^{\p}:
A_d^{\C}(m,n;g) \to F(m,n;g)
\end{cases}
\end{equation}
the natural maps.
\par\vspace{3mm}\par
The notations used in this paper can be summarized in the following 
two diagrams, where $g\in \Alg_d^*(\RP^{m-1},\KP^n)$ denotes a fixed based algebraic map 
of degree $d$ and  
we omit their details of the notations in  (\ref{diagramR})
and refer the reader to \cite{AKY1}.  

\begin{equation}\label{diagramC}
\xymatrix{%
\Map_{[d]_2}(\RP^m,\CP^n)     &     \Map^*_{[d]_2}(\RP^m,\CP^n) \ar@{_{(}->}[l]  &   F^{\C}(m,n;g) \ar@{_{(}->}[l] \\
\Alg_d(\RP^m,\CP^n) \ar@{^{(}->}[u]^{j_{d,\C}}     &     
\Alg_d^*(\RP^m,\CP^n) \ar@{^{(}->}[u]^{i_{d,\C}} \ar@{_{(}->}[l]  &   
\Alg^{\C}_d(m,n;g)  \ar@{^{(}->}[u]^{i_{d,\C}^{\p}} \ar@{_{(}->}[l]\\
\tilde{A}_d^{\C}(m,n) \ar@{->>}[u]^{\Gamma^{\C}_{d}} 
\ar@/^4pc/[uu]^>>>>>>{j_d^{\C}}   &     A_d^{\C}(m,n) \ar@{->>}[u]^{\Psi^{\C}_d} \ar@/^4pc/[uu]^>>>>>>{i_d^{\C}} &   
A_d^{\C}(m,n;g) \ar@{->>}
[u]^{{\Psi_d^{\C}}^{\p}} \ar@/^4pc/[uu]^>>>>>>{i_d^{\p\p}}
}
\end{equation}
\begin{equation}\label{diagramR}
\xymatrix{%
\Map_{[d]_2}(\RP^m,\RP^n)     &     \Map^*_{[d]_2}(\RP^m,\RP^n) \ar@{_{(}->}[l]  &   F(m,n;g) \ar@{_{(}->}[l] \\
\Alg_d(\RP^m,\RP^n) \ar@{^{(}->}[u]^{j_{d,\R}}     &     \Alg_d^*(\RP^m,\RP^n) \ar@{^{(}->}[u]^{i_{d,\R}} \ar@{_{(}->}[l]  &   \Alg_d(m,n;g)  \ar@{^{(}->}[u]^{i_{d,\R}^{\p}} \ar@{_{(}->}[l]\\
\tilde{A}_d(m,n) \ar@{->>}[u]^{\Gamma_d} \ar@/^4pc/[uu]^>>>>>>{j_d}   &     A_d(m,n) \ar@{->>}[u]^{\Psi_d} \ar@/^4pc/[uu]^>>>>>>{i_d} &   A_d(m,n;g) \ar@{->>}[u]^{\Psi_d^{\p}} \ar@/^4pc/[uu]^>>>>>>{i_d^{\p}}
}
\end{equation}




\section{The main results.}\label{section: main}

In this section we state the main results of this paper.
First define the positive integers
$D^*_{\Bbb K}(d;m,n)$ and $D_{\Bbb K}(d;m,n)$
by
\begin{eqnarray}\label{Dnumber}
D_{\K}^*(d;m,n) &=&
\begin{cases}
(n-m)\big(\lfloor \frac{d+1}{2}\rfloor  +1\big) -1
& \mbox{if } \K =\R,
\\
(2n-m+1)\big(\lfloor \frac{d+1}{2}\rfloor  +1\big) -1
& \mbox{if } \K =\C,
\end{cases}
\\
D_{\K}(d;m,n) &=&
\begin{cases}
(n-m)\big(d+1\big) -1
& \mbox{if } \K =\R,
\\
(2n-m+1)\big(d  +1\big) -1
& \mbox{if } \K =\C,
\end{cases}
\end{eqnarray}
where
$\lfloor x\rfloor$ is the integer part of a real number $x$
and we
remark that the equality $D_{\C}(d;m,n)=D_{\R}(d;m,2n+1)$ 
holds.
\par\vspace{2mm}\par
We first recall 
the following two results. 

\begin{thm}[The case $(\K ,m)=(\R ,1)$; \cite{KY1}, \cite{Y5}]
\label{thm: KY1}
If $m=1<n$, the natural map
$i_d:A_d(1,n)\to \Map^*_{[d]_2}(\RP^1,\RP^n)\simeq \Omega S^n$
is a homotopy equivalence up to dimension $D_{\R}(d;1,n)=(n-1)(d+1)-1.$
\qed
\end{thm}

\begin{thm}[The case $\K =\R$ and $m\geq 2$; \cite{KY4}]
\label{thm: KY4}
Let $m$ and $n$ be positive integers such that
$2\leq m<n$.
\begin{enumerate}
\item[$\I$] 
Let $g\in \Alg_d^*(\RP^{m-1},\RP^n)$ be an algebraic map of degree $d$.
Then the natural map
$i_{d}^{\p}:A_d(m,n;g)\to F(m,n;g)\simeq \Omega^mS^n$
is a homotopy equivalence up to dimension $D_{\R}(d;m,n)$ if $m+2\leq n$
and a homology equivalence up to dimension $D_{\R}(d;m,n)$ if $m+1=n$.
\item[$\II$]
The natural maps
$$
\begin{cases}
i_d:A_d(m,n)\to \Map_{[d]_2}^*(\RP^m,\RP^n)
\\
j_d:\tilde{A}_d(m,n)\to \Map_{[d]_2}(\RP^m,\RP^n)
\end{cases}
$$
are homotopy equivalences up to dimension $D_{\R}(d;m,n)$ if
$m+2\leq n$, and  homology equivalences up to dimension
$D_{\R}(d;m,n)$ if $m+1=n$.
\qed
\end{enumerate}
\end{thm}

\begin{rem}\label{Remark: AKY1}
(i)
The above theorem was recently proved in \cite{KY4} and  is an improvement of the main result of \cite{AKY1}.
\par
(ii)
A map $f:X\to Y$ is called {\it a homotopy} (resp. {\it a homology}) {\it equivalence up to dimension} $D$ if
$f_*:\pi_k(X)\to \pi_k(Y)$ (resp. $f_*:H_k(X,\Z)\to H_k(Y,\Z)$) is an isomorphism for any
$k<D$ and an epimorphism for $k=D$.
\par
(iii)
Let $G$ be a finite group and let
$f:X\to Y$ be a $G$-equivariant map.
Then a map $f:X\to Y$ is called 
{\it a $G$-equivariant homotopy }(resp. {\it homology})
{\it equivalence up to dimension }$D$ if the restriction
map
$f^{H}:X^{H}\to Y^{H}$
is a homotopy equivalence (resp. homology) equivalence up to
dimension $D$ for any subgroup $H\subset G$.
\end{rem}

\par\vspace{2mm}\par


Now we state the main results of this paper as follows.

\begin{thm}\label{thm: KY5-I}
Let $m$ and $n$ be positive integers such that $2\leq m\leq 2n$. 
\begin{enumerate}
\item[$\I$]
Let $g\in \Alg_d^{\C}(\RP^{m-1},\CP^n)$ be a fixed algebraic map of 
degree $d$.
The natural map
$i_{d}^{\p\p}:A_d^{\C}(m,n;g)\to F(m,n;g)\simeq \Omega^mS^{2n+1}$
is a homotopy equivalence up to dimension $D_{\C}(d;m,n)$ if
$m<2n$ and a homology equivalence up to dimension $D_{\C}(d;m,n)$
if $m=2n$. 
\item[$\II$]
The natural maps
$$
\begin{cases}
i_d^{\C}:A_d^{\C}(m,n)\to \Map^*_{[d]_2}(\RP^m,\CP^n)
\\
j_d^{\C}:\tilde{A}_d^{\C}(m,n)\to \Map_{[d]_2}(\RP^m,\CP^n)
\end{cases}
$$
are homotopy equivalences up to dimension $D_{\C}(d;m,n)$
if 
$m<2n$ and  homology equivalences up to dimension $D_{\C}(d;m,n)$
if $m=2n$.
\end{enumerate}
\end{thm}
\par\vspace{2mm}\par
Note that the
complex conjugation on $\C$ naturally induces $\Z/2$-actions on
the spaces $\Alg_d^{\C}(m,n;g)$ and $A_d^{\C}(m,n)$.
In the same way, it also induces a $\Z/2$-action on $\CP^n$ and 
it extends to $\Z/2$-actions on the spaces
$\Map^*(\RP^m,S^{2n+1})$ and $\Map^*_{\epsilon}(\RP^m,\CP^n)$,
 where we identify 
$S^{2n+1}=\{(w_0,\cdots ,w_n)\in\C^{n+1}:\sum_{k=0}^n\vert w_k\vert^2 =1\}$
and regard $\RP^m$ as a $\Z/2$-space with the trivial $\Z/2$-action.
\par\vspace{3mm}\par
Since $(i_d^{\p\p})^{\Z/2}=i_d^{\p}$, $(i_d^{\C})^{\Z/2}=i_d$ and $(j_d^{\C})^{\Z/2}=j_d$,
by  Theorem \ref{thm: KY4} and Theorem \ref{thm: KY5-I} we obtain the following
result.  

\begin{crl}\label{cor: KY5-II}
Let  $m$ and $n$  be positive integers such that
$2\leq m<n$.
\begin{enumerate}
\item[$\I$]
Let $g\in \Alg_d^*(\RP^{m-1},\CP^n)$ be an algebraic map of   degree $d$.
The natural map
$i_{d}^{\p\p}:A_d^{\C}(m,n;g)\to F(m,n;g)\simeq \Omega^mS^{2n+1}$
is a $\Z/2$-equivariant homotopy equivalence
 up to dimension $D_{\R}(d;m,n)$ if
$m+2\leq n$ and a $\Z/2$-equivariant homology equivalence up to dimension $D_{\R}(d;m,n)$
if $m+1=n$. 
\item[$\II$]
The natural maps
$$
\begin{cases}
i_d^{\C}:A_d^{\C}(m,n)\to \Map^*_{[d]_2}(\RP^m,\CP^n)
\\
j_d^{\C}:\tilde{A}_d^{\C}(m,n)\to \Map_{[d]_2}(\RP^m,\CP^n)
\end{cases}
$$
are
$\Z/2$-equivariant homotopy equivalences up to dimension $D_{\R}(d;m,n)$
if
$m+2\leq n$ and  are $\Z/2$-equivariant homology equivalences up to dimension 
$D_{\R}(d;m,n)$
if $m+1=n$.
\qed
\end{enumerate}
\end{crl}

\begin{rem}\label{Remark: proof}
Remark that $D_{\K}^*(d;m,n)< D_{\K}(d;m,n)$ for $d\geq 2$.
Thus,
we may regard 
Theorem \ref{thm: KY5-I} and Corollary \ref{cor: KY5-II}
as
the improvements of  
\cite[Theorems 1.4, 1.5 and Corollary 1.7]{KY3}.
The method of the proof of Theorem \ref{thm: KY5-I} consists of applying the ideas used in 
the proof of Theorem \ref{thm: KY4} to the argument of the proof of
\cite[Theorems 1.4 and 1.5]{KY3}. 
We shall therefore omit the details and refer the reader to \cite{KY3} and \cite{KY4}.
\end{rem}

\begin{rem}
There is a mistake in the statement of  \cite[Corollary 1.7]{KY3};  the condition $2\leq m\leq 2n$ should be
replaced by  $2\leq m<n$ as in Corollary \ref{cor: KY5-II}.
\end{rem}

\section{The projections $\Psi_d$ and $\Gamma_d$.}
\label{section: projection}

In this section, we consider the following conjecture stated in \cite{AKY1}.
\par\vspace{2mm}
\begin{conj}[\cite{AKY1}, Conjecture 3.8]\label{conj}
{\it If $1\leq m<n$, the  projection maps
$\Psi_d:A_d(m,n)\to \Alg_d^*(\RP^m,\RP^n)$ and
$\Gamma_d:\tilde{A}_d(m,n)\to\Alg_d(\RP^m,\RP^n)$ are homotopy equivalences.}
\end{conj}
It is easy to see that $\Gamma_d$ is a homotopy equivalence if $\Psi_d$ is so.
Thus, we only consider the projection $\Psi_d$.
There is much evidence that suggests that $\Psi_d$ is a homotopy equivalence.
Indeed, the fibre of it over $\Alg_{d-2k+2}^*(\RP^m,\RP^n)\setminus
\Alg_{d-2k}^*(\RP^m,\RP^n)$ is homeomorphic to the space of everywhere positive
$\R$-coefficient polynomials in $(m+1)$-variables of degree $2k$ with leading coefficient $1$,
which is convex and contractible.
If it is a quasi-fibration, it is a homotopy equivalence.
Although we cannot prove this in general, we can do so it for $m=1$.
\begin{thm}\label{thm: conj}
If $m=1$, Conjecture \ref{conj} 
is true.
\end{thm}
\noindent{\it Proof. }
We will exploit the convenient fact that spaces of tuples of polynomials in one variable can be identified with certain configuration spaces of points or particles in the complex plane. More exactly, we identify the space $A_d(1,n)$ with the space of $d$ particles of each of $n+1$ different colours, in the case $n=2$, say, red, blue and yellow, located in $\Bbb C = \Bbb R^2$ symmetrically with respect to the real axis and such that no three particles of different colour lie at the same point on the real axis. 

\begin{center}
\includegraphics[width=3.2cm]{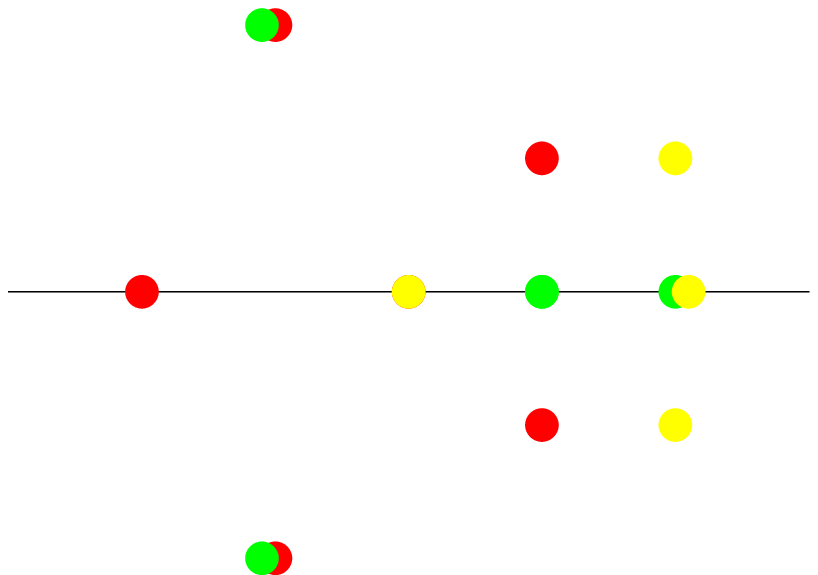} 
\end{center}
Note that off the real axis the particles are completely unrestricted. The space of algebraic maps $A_d(1,n)$ can also be thought of as a configuration space of  $k$ particles of each of $n+1$ different colours as above, where $k\leq d$, but with the additional property that when $n+1$ different particles meet (off the real line) they (and their conjugate particles) disappear (this space is topologised as the obvious quotient of the preceding one). 

Finally, we need one more configuration space, introduced in \cite{Mo1}.  Let $T(d,n)$ denote the space of  no more than $k_i\le d$ particles of colour $i$, where $k_i=d$ mod $2$, on the real axis , with the property that any even number of particles of the same colour at the same point on the real axis vanish, and  of course, as before no $n+1$ particles of different colours lie at the same point. In other words,  $T(d,n)$ is a configuration space modulo 2. 
There is a natural map $\Phi :A_d(1,n)\to T(d,n)$ which factors through 
$\Psi_d$,
$\Phi =Q_d\circ \Psi_d:A_d(1,n)
\stackrel{\Psi_d}{\longrightarrow} \Alg_d^*(\RP^1,\RP^n) \stackrel{Q_d}{\longrightarrow} T(d,n).$

 \begin{prp}[\cite {Mo1}, Proposition 2.1]
\label{prop: Mo}
The maps $\Psi_d$ and $Q_d$ above are homotopy equivalences. 
\end{prp}

A proof of this proposition is given in \cite{Mo1} but, as we stated earlier,   it does not seem convincing  to us, so we will give here a different one. 
More precisely, we need the following:

 \begin{lmm}
\label{lemma: K}
The maps $Q_d$ and $Q_d \circ \Psi_d$  are quasi-fibrations with contractible fibres. 
\end{lmm}
From this it follows at once that $\Psi_d$ is a homotopy equivalence, and
this completes the proof of Theorem \ref{thm: conj}.
\qed

\begin{proof}[Proof of Lemma \ref{lemma: K}]
We first prove that the fibre of  $Q_d \circ\Psi_d$ (and $Q_d$) over any point in $T(d,n)$ is contractible. Consider a configuration in $T(d,n)$.
\begin{center}
\begin{picture}(80,20)
\put(20,5){\line(1,0){100}}
\put(30,2.5){\textcolor{red}{$\bullet$}}
\put(59,2.5){\textcolor{yellow}{$\bullet$}}
\put(86,2.5){\textcolor{green}{$\bullet$}}
\put(89,2.5){\textcolor{yellow}{$\bullet$}}
\end{picture}
\end{center}

In the fibre over this configuration, all points in the upper half plane  are sent linearly to the fixed point $(1,0)$ and those in the lower half plane to $(-1,0)$. For a $2k$ or $2k+1$ fold particle lying on the real line, $k$ of the particles are moved to ${1,0}$ and $k$ are moved to  to $(-1,0)$ leaving $0$ or $1$ particles in place. 
This argument shows that the fibres of both  $Q_d\circ \Psi_d$ and $Q_d$ are contractible. 

\begin{center}
\includegraphics[width=4cm]{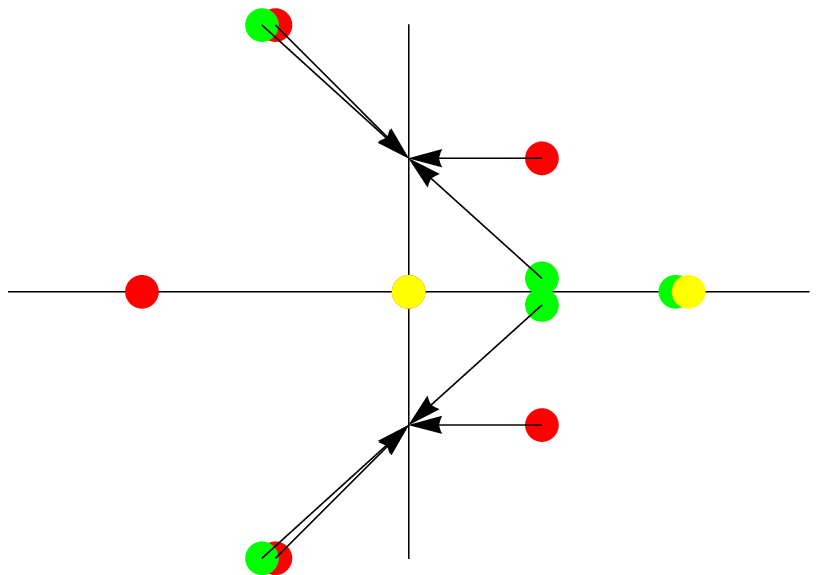} 
\end{center}
\par
Next, we will show that the maps  $Q_d \circ\Psi_d$ and $Q_d$ are both quasi-fibrations. Consider a point (configuration of particles)  in $T(d,n)$. By a \lq singular sub-configuration  for colour $i$\rq\  we mean a collection $n$ particles of distinct colours other than the $i$-th colour which are all located at the same point.  Intuitively, we think of a singular configuration for the $i$-th colour as an obstacle through which a particle of the $i$-th colour cannot pass. By a singular sub-configuration (without specifying the colour) we mean a singular sub-configuration for any colour. 

Clearly, it is easy to see that the topology of the fibre over a particular configuration in $T(d,n)$ is determined by two quantities: the number of particles of each colour in  $T(d,n)$ and the number of singular sub-configurations (for each colour). Hence, it the proof we will consider filtrations with respect to the number of particles of each colour and the number of singular sub-configurations.  

For non-negative integers $p_1,p_2,\dots, p_{n+1}$ with $\sum _i p_i\le d$, let  $T(d,n;p_1,\dots,p_{n+1})$ denote the subspace of $T(d,n)$ consisting of  configurations with precisely $p_i$ particles of  the $i$-th colour. We start by proving that the restrictions of the maps $Q_d\circ \Psi_d$ and $Q_d$ to the pre-images of $T(d,n;p_1,\dots,p_{n+1})$ are quasi-fibrations.
 For any integer $k\ge 0$, let $T_k(d,n;p_1,\dots,p_{n+1})$ denote the subspace of $T(d,n;p_1,\dots,p_{n+1})$ consisting of configurations containing exactly $k$ singular sub-configurations. It is easy to show that the restriction of $Q_d \circ\Psi_d$ and $Q_d$ to the pre-image of $T_k(d,n;p_1,\dots,p_{n+1})$ is a locally trivial fibre bundle. Now we filter the space $T(d,n;p_1,\dots,p_{n+1})$ by closed subspaces $D_k(d,n;p_1,\dots,p_{n+1})$ of configurations containing $\ge k$ singular sub-configurations. 
 Set theoretic differences between these spaces are the spaces $T_k(d,n;p_1,\dots,p_{n+1})$ over which the maps are locally trivial fibre bundles and hence quasifibrations. Now we apply the Dold-Thom criterion \cite[Lemma 4.3]{Hatch}. To do so we have to construct an open neighbourhood of $D_{k+1}(d,n;p_1,\dots,p_{n+1})$ in $D_{k}(d,n;p_1,\dots,p_{n+1})$, a deformation of this neighbourhood onto $D_{k+1}(d,n;p_1,\dots,p_{n+1})$, together with a corresponding covering neighbourhood and a covering deformation required by the Dold-Thom criterion (since all the fibres are contractible the condition that the induced maps on the fibres are homotopy equivalences is automatically satisfied). Such neighbourhoods and deformations are easy to describe intuitively. The set of points of the required  neighbourhood of  $D_{k+1}(d,n;p_1,\dots,p_{n+1})$ consists of the points of  $D_{k+1}(d,n;p_1,\dots,p_{n+1})$ together with those configurations in $D_{k}(d,n;p_1,\dots,p_{n+1})$ with at least one non-singular sub-configuration of $n$ particles of different colours  contained in a \lq sufficiently small\rq\  interval. Here \lq sufficiently small\rq\  refers to the requirement that the particles in this sub-configuration be much nearer to each other than they are to any other particle and that the length of the minimal interval containing the sub-configuration be much less than the length of any interval containing a collection of $n+1$ particles of different colour.  The deformation can now be defined by introducing a force of attraction between particles of different colour (e.g. a force field satisfying  an inverse-square law). 
By induction on $k$ we show that the maps  $Q_d \circ\Psi_d$ and $Q_d$ are quasifibrations over $T(d,n;p_1,\dots,p_n)$. We now fix $p_1,p_2,\dots, p_n$ and filter the space $T(d,n)$ according to the number of points of the $n+1$-th colour. The set theoretic differences between the terms of the filtration are precisely the spaces $T(d,n;p_1,\dots,p_{n+1})$ and we have already proved that the restriction of the maps $Q_d\circ \Psi_d$ and $Q_d$ to the inverse images of these spaces are quasi-fibrations. We apply again the Dold-Thom criterion. For this purpose we need to construct open neighbourhoods of spaces of configurations  with no more than $k-2$ particles of the last colour in the space of configurations of no more than $k$-particles of the last colour. The method is again analogous to the one we used earlier. Our deformation will pull together pairs of particles of the last colour which are very close by means of a gravitational force field between particles of the last colour. For this purpose we must avoid hitting a singular sub-configuration. Again, it is easy to see that we can choose open neighbourhoods and deformations with the right properties. This proves that the maps restricted to the pre-images of spaces with the number of particles of the first $n-1$ colours fixed are quasifibrations. Now we filter these spaces according to the number of particles of the last but one colour. Proceeding by induction we see that $Q_d \circ\Psi_d$ and $Q_d$ are quasi-fibrations.
\end{proof}

\end{document}